\documentclass[11pt]{amsart}
\usepackage{pdfsync}

\author[A.~J.~Di Scala]{Antonio J.~Di Scala$^*$}
\thanks{$^*$A.~J. Di Scala is member of GNSAGA of INdAM and of DISMA Dipartimento di Eccellenza MIUR 2018-2022}
\address{Politecnico di Torino, Dipartimento di Scienze Matematiche, Torino, Italy}
\email{antonio.discala@polito.it}

\author[N.~Murru]{Nadir Murru}
\address{Universit\`a degli Studi di Torino, Department of Mathematics, Torino, Italy}
\email{nadir.murru@unito.it}

\author[C.~Sanna]{Carlo Sanna$^\dagger$}
\thanks{$\dagger\,$C.~Sanna is supported by a postdoctoral fellowship of INdAM and is a member of the INdAM group GNSAGA}
\address{Universit\`a degli Studi di Genova, Department of Mathematics, Genova, Italy}
\email{carlo.sanna.dev@gmail.com}


\title{Lucas pseudoprimes and the Pell conic}

\usepackage{amsmath}
\usepackage{amssymb}
\usepackage{amsthm}
\usepackage{geometry}
\usepackage{hyperref}
\usepackage{enumerate}
\hypersetup{colorlinks=true}
\usepackage[utf8]{inputenc}

\newtheorem{thm}{Theorem}[section]
\newtheorem{lem}[thm]{Lemma}

\theoremstyle{remark}
\newtheorem{rmk}[thm]{Remark}
\newtheorem{exm}[thm]{Example}

\begin{document}

\begin{abstract}
We show a connection between the Lucas pseudoprimes and the Pell conic equipped with the Brahmagupta product introducing the Pell pseudoprimes.
\end{abstract}

\maketitle

\section{Introduction}
Primality testing is a very important topic, especially in cryptography (see, e.g.,~\cite{Yan04} for an overview). 
The most classical primality tests are: Fermat and Euler's test~\cite{Rib04}, Lucas test~\cite{BW80}, Solovay--Strassen primality test~\cite{Sol77}, Rabin--Miller primality test~\cite{Mil76, Rab80}, Baillie--PSW primality test~\cite{BW80}, and AKS primality test~\cite{AKS04}.

The Lucas test is based on some properties of Lucas sequences.
Given two integers $P > 0$ and $Q$ such that $D := P^2 - 4Q \neq 0$, the Lucas sequences $(U_k)_{k \geq 0}$ and $(V_k)_{k \geq 0}$ associated to $(P,Q)$ are defined by
\begin{center}
\begin{minipage}{.3\linewidth}
\begin{align*}
U_0 &= 0,\\
U_1 &= 1, \\
U_k &= P U_{k-1} - Q U_{k-2},
\end{align*}
\end{minipage}%
\begin{minipage}{.3\linewidth}
\begin{align*}
V_0 &= 2,\\
V_1 &= P,\\
V_k &= P V_{k-1} - Q V_{k-2} ,
\end{align*}
\end{minipage}\\[1em]
\end{center}
for every integer $k \geq 2$.
We will write $U_k(P,Q)$ and $V_k(P,Q)$ when it is necessary to show the dependency on $P$ and $Q$.
The Lucas test is based on the fact that when $n$ is an odd prime with $\gcd(n, Q) = 1$, we have
\begin{equation}\label{equ:lucaspseudoprime}
U_{n-\big(\tfrac{D}{n}\big)} \equiv 0 \pmod n ,
\end{equation}
where $\big(\tfrac{D}{n}\big)$ denotes the Jacobi symbol.
If $n$ is composite but~\eqref{equ:lucaspseudoprime} still holds, then $n$ is called a \emph{Lucas pseudoprime} with parameters $P$ and $Q$~\cite{BW80}.
Lucas pseudoprimes have been widely studied~\cite{DF88, GP91, Som09, Suw12}.
Some authors also studied primality tests using more general linear recurrence sequences~\cite{MR866094, Gra10, MR1035934}.

In this paper, we highlight how the Lucas test can be introduced in an equivalent way by means of the Brahmagupta product and the Pell's equation.

The \emph{Pell's equation} is the Diophantine equation
\begin{equation*}
x^2 - D y^2 = 1 ,
\end{equation*}
where $D$ is a fixed squarefree integer.
It is well known that given two solutions $(x_1, y_1)$ and $(x_2, y_2)$ the Brahmagupta product
\begin{equation*}
 (x_1, y_1) \otimes_D (x_2, y_2) := (x_1 x_2 + D y_1 y_2, x_1 y_2 + x_2 y_1)
\end{equation*}
yields another solution of the Pell's equation (see, e.g.,~\cite{Bar03}).
Given a ring $\mathcal R$, we can consider the Pell conic
\begin{equation*}
 \mathcal C = \mathcal{C}_D(\mathcal{R}) := \{ (x,y) \in \mathcal{R} \times \mathcal{R} : x^2 - D y^2 = 1 \} .
\end{equation*}
In particular, if $\mathcal R$ is a field then $(\mathcal C, \otimes_D)$ is a group with identity $(1,0)$.
Moreover, when $\mathcal R = \mathbb Z_p$ (the field of residue classes modulo a prime number $p$), we have $\lvert \mathcal C \rvert = p - \big(\tfrac{D}{p}\big)$ (see, e.g.,~\cite{MV92}).
Consequently, when $n$ is an odd prime, we have
\begin{equation} \label{eq:pell}
(\widetilde x, \widetilde y)^{\otimes n - \big(\tfrac{D}{n}\big)} \equiv (1, 0) \pmod n,
\end{equation}
for all $(\widetilde x, \widetilde y) \in \mathcal C_D(\mathbb{Z}_p)$, where the power is evaluated with respect to $\otimes := \otimes_D$.
We say that an odd composite integer $n$ such that $\gcd(n, \widetilde{y}) = 1$ and
\begin{equation*}
y_n \equiv 0 \pmod n ,
\end{equation*}
where $(x_n, y_n) = (\tilde x, \tilde y)^{\otimes n - (D/n)}$, is a \emph{Pell pseudoprime} with parameters $D$ and $(\widetilde{x}, \widetilde{y}) \in \mathcal{C}_D(\mathbb{Z}_n)$.
The possibility of using the properties of the Pell conic for constructing a primality test has been highlighted in~\cite{Lem03}, but the author does not provide an extensive study about it and only focuses on a conic of the kind $x^2 - Dy^2 = 4$.
Moreover, in~\cite{Ham12} the author used the conic $x^2 + 3 y^2 = 4$ for testing numbers of the form $3^n h \pm 1$.

\begin{rmk}
The term \emph{Pell pseudoprime} is already used for the Lucas pseudoprimes with parameters $P = 2$ and $Q = -1$.
Indeed, in this case, the sequence $U_n$ is known as the Pell sequence (A000129 in OEIS~\cite{Slo}).
Furthermore, the term Pell pseudoprime is also used to indicate a composite integer $n$ such that
\begin{equation*}
 U_n \equiv \left( \cfrac{2}{n} \right) \pmod n
\end{equation*}
for $P = 2$ and $Q = -1$ (A099011 in OEIS).
\end{rmk}

The relation between Lucas pseudoprimes and Pell pseudoprimes is given by the following result.

\begin{thm}\label{thm:main}
On the one hand, if $n$ is a Lucas pseudoprime with parameters $P > 0$ and $Q = 1$, then $n$ is a Pell pseudoprime with parameters $D = P^2 - 4$ and $(\widetilde x, \widetilde y) \equiv (2^{-1}P, 2^{-1}) \pmod n$. 
On the other hand, if $n$ is a Pell pseudoprime with parameter $D$ and $(\widetilde x, \widetilde y)$, then $n$ is a Lucas pseudoprime with parameters $P = 2 \widetilde x$, and $Q = 1$.
\end{thm}

\section{Proof of Theorem~\ref{thm:main}}

\begin{lem}\label{lem:bra}
Let $\widetilde{x}, \widetilde{y} \in \mathbb{Z}$ and let $D$ be a nonzero integer.
We have
\begin{equation*}
(\widetilde{x}, \widetilde{y})^{\otimes k} = \big(\tfrac1{2}V_k(P, Q), \widetilde{y}\, U_k(P,Q)\big) ,
\end{equation*}
for every integer $k \geq 0$, where $P := 2\widetilde{x}$, $Q := \widetilde{x}^2 - D\widetilde{y}^2$, and the Brahmagupta product is computed respect to $D$.
\end{lem}
\begin{proof}
It is clear from the definition of Brahmagupta product that $(\widetilde{x}, \widetilde{y})^{\otimes k} = (x_k, y_k)$, where $x_k, y_k \in \mathbb{Z}$ are defined by $(\widetilde{x} + \sqrt{D}\widetilde{y})^k = x_k + \sqrt{D} y_k$.
Conjugating this last equality we get $(\widetilde{x} - \sqrt{D}\widetilde{y})^k = x_k - \sqrt{D} y_k$, from which in turn we obtain
\begin{equation*}
x_k = \frac{(\widetilde{x} + \sqrt{D}\widetilde{y})^k + (\widetilde{x} - \sqrt{D}\widetilde{y})^k}{2} \quad\text{ and}\quad y_k = \frac{(\widetilde{x} + \sqrt{D}\widetilde{y})^k - (\widetilde{x} - \sqrt{D}\widetilde{y})^k}{2\sqrt{D}} .
\end{equation*}
It is well known~\cite[Ch.~1]{MR1761897} that 
\begin{equation*}
V_k(P, Q) = \alpha^k + \beta^k \quad\text{ and }\quad U_k = \frac{\alpha^k - \beta^k}{\alpha - \beta} ,
\end{equation*}
where $\alpha, \beta$ are the roots of $X^2 - PX + Q$.
Since $P := 2\widetilde{x}$ and $Q := \widetilde{x}^2 - D\widetilde{y}^2$, we have $\alpha = \widetilde{x} + \sqrt{D}\widetilde{y}$ and $\beta = \widetilde{x} - \sqrt{D}\widetilde{y}$, so that $x_k = \tfrac1{2}V_k(P,Q)$ and $y_k = \widetilde{y} U_k(P, Q)$, as desired.
\end{proof}

Suppose that $n$ is a Lucas pseudoprime with parameters $P > 0$ and $Q = 1$.
Let $\widetilde{x}, \widetilde{y} \in \mathbb{Z}$ be such that $(\widetilde{x}, \widetilde{y}) \equiv (2^{-1} P, 2^{-1}) \pmod n$ and put $D := P^2 - 4$.
We have
\begin{equation*}
\widetilde{x}^2 - D\widetilde{y}^2 \equiv (2^{-1}P)^2 - (P^2 - 4)2^{-2} \equiv 1 \pmod n ,
\end{equation*}
so that $(\widetilde{x}, \widetilde{y}) \in \mathcal{C}_D(\mathbb{Z}_n)$.
Moreover, by Lemma~\ref{lem:bra} with $k = n - (D|n)$ and since $n$ is a Lucas pseudoprime, we have
\begin{equation*}
(\widetilde{x}, \widetilde{y})^{\otimes k} = \big(\tfrac1{2}V_k(P, Q), \widetilde{y}\, U_k(P,Q)\big) \equiv \big(\tfrac1{2}V_k(P, Q), 0\big) \pmod n .
\end{equation*}
Hence, $n$ is a Pell pseudoprime with parameters $D = P^2 - 4$ and $(\widetilde{x}, \widetilde{y}) \equiv (2^{-1}P, 2^{-1}) \pmod n$.

Now suppose that $n$ is Pell pseudoprime with parameters $D$ and $(\widetilde{x}, \widetilde{y}) \in \mathcal{C}$.
Let $P = 2\widetilde{x}$ and $Q = 1$.
Note that since $n$ is Pell pseudoprime, by definition we have $\gcd(\widetilde{y}, n) = 1$.

By Lemma~\ref{lem:bra} with $k = n - (D|n)$ and since $n$ is a Pell pseudoprime, we have
\begin{equation*}
\big(\tfrac1{2}V_k(P, Q), \widetilde{y}\, U_k(P,Q)\big) \equiv (\widetilde{x}, \widetilde{y})^{\otimes k} ,
\end{equation*}
so that $U_k(P,Q) \equiv 0 \pmod n$.
Hence, $n$ is a Lucas pseudoprime with paraments $P = 2\widetilde{x}$ and $Q = 1$.

\section{Further remarks}

Let us note that, fixed the parameters $P$ and $Q = 1$ for the Lucas test (for checking, e.g., the primality of all the integers in a certain range), there is not a corresponding Pell test with fixed parameters $D$, $\widetilde{x}$ and $\widetilde{y}$ as integer numbers.
Indeed, given any $P$ and $Q = 1$, we have seen that $D = P^2 - 4$ and $(\widetilde{x}, \widetilde{y}) \equiv (2^{-1}P, 2^{-1}) \pmod n$ are the corresponding parameters of the Pell test, but these values depend on the integer $n$ we are testing.

Moreover, in general, we are not able to fix the integer parameters $D, \widetilde{x}, \widetilde{y}$ in the Pell test for checking the primality of all the integers in a given range, because it is necessary that $\widetilde{x}^2 - D \widetilde{y}^2 \equiv 1 \pmod n$ and this can not be true for any integer $n$.
For overcoming these issues, the use of a parametrization of the conic $\mathcal C$ can be helpful.
In~\cite{BM16}, the authors provided the following map
\begin{equation*}
 \Phi : \begin{cases} \mathcal R \cup \{\alpha\} \rightarrow \mathcal C \cr a \mapsto \left( \cfrac{a^2 + D}{a^2 - D}, \cfrac{2 a}{a^2 - D} \right) \end{cases}
\end{equation*}
where $\alpha \not \in \mathcal R$ is the point at the infinity of such a parametrization of $\mathcal C$.
When $\mathcal R$ is a field and $t^2 - D$ is irreducible in $\mathcal R$, the map is always defined, otherwise there are values of $a$ such that $\Phi(a)$ can not be evaluated.
In this way, we can consider the Pell test with fixed parameters $D$ and $a$, in the sense that $\widetilde{x} = (a^2 + D)/(a^2 - D)$ and $\widetilde{y} = 2 a/(a^2 - D)$.
However, given a Pell test with parameters $D$ and $a$, there is not always a corresponding Lucas test with fixed parameters $P$ and $Q = 1$ as integer numbers.
Indeed, the correspondence is given by considering $P = 2 \widetilde{x}$.
We see some examples for clarifying these considerations.

\begin{exm} \label{es:1}
Fixed $P = 3$ and $Q = 1$, the first Lucas pseudoprime is 21, indeed we have $\left( \cfrac{5}{21} \right) = 1$ and
\begin{equation*}
 U_{20} = 102334155 \equiv 0 \pmod{21}.
\end{equation*}
It is also a Pell pseudoprime for $D = P^2 - 4 = 5$, $\widetilde{x} = P/2 \pmod {21} = 12$, $\widetilde{y} = 1/2 \pmod {21} = 11$, indeed
\begin{equation*}
 (12, 11)^{\otimes 20} \equiv (13, 0) \pmod{21}.
\end{equation*}
The second Lucas pseudoprime, in this case, is 323 and it is a Pell pseudoprime for $D = P^2 - 4 = 5$, $\widetilde{x} = P/2 \pmod 323 = 163$, $\widetilde{y} = 1/2 \pmod 323 = 162$, which are different from the previous parameters (the point $(163, 162)$ does not belong to $\mathcal C$ for $D = 5$ and $R = \mathbb Z_{21}$).
\end{exm}

\begin{exm}
If we consider the parameters $D = 3$, $\widetilde{x} = 8$, $\widetilde{y} = 66$, in the interval $[1, 100]$, we can only test the integers $3, 5, 9, 15, 17, 45, 51, 85$, since for the other integers $m \in [1, 100]$ the oint $(8, 65)$ does not belong to $\mathcal C$ for $D = 5$ and $R = \mathbb Z_{m}$. For instance, we can test the integer $n = 85$ and observing that it is a Pell pseudoprime in this case, consequently it is a Lucas pseudoprime for $P = 16$ and $Q = 1$.
Let us note that $\left( \cfrac{D}{n} \right) = \left( \cfrac{3}{85} \right) = 1$ and $\left( \cfrac{P^2 - 4 Q}{n} \right) = \left( \cfrac{252}{85} \right) = 1$. \\
On the other hand, we can test the integer $85$ with the Pell using different parameters, e.g., $D = 3$, $x_1 = 7$, $y_1 = 4$ (being $(7, 4) \in \mathcal C$ in this case) and we have
\begin{equation*}
(7, 4)^{\otimes 84} \equiv (76, 15) \pmod{85},
\end{equation*}
i.e., $n$ is not a Pell pseudoprime. The corresponding Lucas test is given by $P = 14$ and $Q = 1$ and we get that 85 in not a Lucas pseudoprime, since 
\begin{equation*}
U_{84} \equiv 25 \pmod 85.
\end{equation*}

\end{exm}

\begin{exm} \label{exm:P-even}
Given $P = 4$ and $Q = 1$, the Lucas pseudoprimes up to 5000 are
\begin{equation*}
 65, 209, 629, 679, 901, 989, 1241, 1769, 1961, 1991, 2509, 2701, 2911, 3007, 3439, 3869.
\end{equation*}
When $P$ is even, we are always able to find an equivalent Pell test, providing all the same pseudoprimes of the Lucas test.
Indeed, it is sufficient to choice $D$ and $a$ such that $(a^2 + D) / (a^2 - D)$ is the integer number $P/2$.
For instance in this case, taking $D = 3$ and $a = 3$, we have $\widetilde{x} = 2$ and $\widetilde{y} = 1$.
\end{exm}

\begin{exm}
Given $P = 3$ and $Q = 1$ the Lucas pseudoprimes up to 5000 are
\begin{equation*}
 21, 323, 329, 377, 451, 861, 1081, 1819, 2033, 2211, 3653, 3827, 4089, 4181.
\end{equation*}
Also for $P$ odd, we are always able to find an equivalent Pell test exploiting the above parametrization.
In this case, we search for $a$ and $D$ integers such that $(a^2 + D) / (a^2 - D)$ is equal to the fraction $P/2$.
For instance, considering $D = 5$ and $a = 5$, we have $\widetilde{x} = 3/2$ and $\widetilde{y} = 1/2$. Let us note that in this case the values of $\widetilde{x}$ and $\widetilde{y}$ will be different as integer numbers, depending on the integer we are testing.
\end{exm}

\begin{exm}
Let us see some Pell tests that do not have an equivalent Lucas test with fixed integer parameters.
Given $D = 6$ and $a = 4$, the Pell pseudoprimes up to 3000 are
\begin{equation*}
 77, 187, 217, 323, 341, 377, 1763, 2387,
\end{equation*}
for this test we have $\widetilde{x} = 11/5$ and $\widetilde{y} = 4/5$.

Given $D = 23$ and $a = 32$, the Pell pseudoprimes up to 3000 are
\begin{equation*}
323, 1047, 1247, 1745, 2813,
\end{equation*}
for this test we have $\widetilde{x} = 1047/1001$ and $\widetilde{y} = 64/1001$.

Given $D = 21$ and $a = 49$, the Pell pseudoprimes up to 3000 are
\begin{equation*}
253, 473, 779, 2627, 2641,
\end{equation*}
for this test we have $\widetilde{x} = 173/170$ and $\widetilde{y} = 7/170$.

Given $D = 29$ and $a = 48$, the Pell pseudoprimes up to 3000 are
\begin{equation*}
989, 1101, 1457, 1991, 2449, 2679
\end{equation*}
for this test we have $\widetilde{x} = 2333/2275$ and $\widetilde{y} = 96/2275$.

\end{exm}

\begin{rmk}
The Lucas test with parameters $P$ and $Q = 1$ is equivalent to the Pell test with parameters $D = P^2 - 4$ and $a = P + 2$.
Indeed, in this case, exploiting the parametrization $\Phi$, we get $x_1 = P/2$ and $y_1 = 1/2$.
Note that using this method, the Pell test equivalent to the Lucas test considered in Example \ref{exm:P-even} has parameters $D = 12$ and $a = 6$. This means that there are Pell tests with different parameters which are equivalent to each others.
\end{rmk}

\begin{rmk}
Considering the identity \eqref{eq:pell}, it is possible to define a stronger test. Indeed some Pell pseudoprimes may not satisfy \eqref{eq:pell} as in example \eqref{es:1} for the Pell pseduoprime 21. The test determined by \eqref{eq:pell} has an equivalent formulation by means of the Lucas sequences. In this case, we can define a pseudoprime as an odd integer $n$ such that
\begin{equation*}
U_{n-\big(\tfrac{D}{n}\big)} \equiv 0 \pmod n \quad \text{and} \quad U_{n-\big(\tfrac{D}{n}\big) + 1} \equiv 1 \pmod n
\end{equation*}
where as usual $D = P^2 - 4 Q$ and $P, Q$ parameters that define the Lucas sequence. This test does not appear in literature with a specific name, but when $D$ is chosen with the Selfridge method~\cite{BW80}, the sequence of pseudoprimes coincides with the Frobenius pseudoprimes with respect to the Fibonacci polynomial~\cite{Gra01}.
\end{rmk}

%
%
%
%


\begin{thebibliography}{21}

\bibitem{MR866094}
W.~W. Adams, \emph{Characterizing pseudoprimes for third-order linear recurrences}, Math. Comp. \textbf{48} (1987), no.~177, 1--15.

\bibitem{AKS04}
M.~Agrawal, N.~Kayal, N.~Saxena, \emph{Primes is in P}, Annals of Mathematics \textbf{160} (2004), 781--793.


\bibitem{BW80}
R.~Baillie, S.~S. Wagstaff, \emph{Lucas pseudoprimes}, Math. Comp. \textbf{35} (1980), no.~152, 1391--1417.

\bibitem{Bar03}
E.~J. Barbeau, \emph{Pell's equation}, New York-Berlin: Springer-Verlag (2003).

\bibitem{BM16}
E.~Bellini, N.~Murru, \emph{An efficient and secure RSA-like cryptosystem exploiting Rédei rational functions over conics}, Finite Fields Appl. \textbf{39} (2016), 179--194.


\bibitem{DF88}
A.~Di Porto, P.~Filipponi, \emph{A probabilistic primality test based on the properties of certain generalized Lucas numbers}, Advances in Cryptology — EUROCRYPT '88, Lecture Notes in Computer Science \textbf{330} (1988), 211--223.

\bibitem{GP91}
D.~M. Gordon, C.~Pomerance, \emph{The distribution of Lucas and elliptic pseudoprimes}, Math. Comp. \textbf{57} (1991), no.~196, 825--838.

\bibitem{Gra01}
J.~Grantham, \emph{Frobenius pseudoprimes}, Math. Comp. \textbf{70} (2001), 873--891-

\bibitem{Gra10}
J.~Grantham, \emph{There are infinitely many Perrin peudoprimes}, J. Number Theory \textbf{130} (2010), 1117--1128.

\bibitem{MR1035934}
S.~Gurak, \emph{Pseudoprimes for higher-order linear recurrence sequences},
  Math. Comp. \textbf{55} (1990), no.~192, 783--813.

\bibitem{Ham12}
S.~A. Hambleton, \emph{Generalized Lucas-Lehmer tests using Pell conics}, Proc. Amer. Math. Soc. \textbf{140} (2012), no.~8, 2653--2661.

\bibitem{Lem03}
F.~Lemmermeyer, \emph{Conics - a poor's man elliptic curves}, Preprint (2003), Available at https://arxiv.org/abs/math/0311306.

\bibitem{Mil76}
G.~L. Miller, \emph{Riemann's hypothesis and tests for primality}, J. Comput. System Sci. \textbf{13} (1976), no.~3, 300--317.

\bibitem{MV92}
A.~J. Menezes, S.~A. Vanstone, \emph{A note on cyclic groups, finite fields, and the discrete logarithm problem}, Appl. Algebra Engrg. Comm. Comput. \textbf{3} (1992), 67--74.



\bibitem{Rab80}
M.~O. Rabin, \emph{Probabilistic algorithm for testing primality}, J. Number Theory \textbf{12} (1980), no.~1, 128--138.

\bibitem{MR1761897}
P.~Ribenboim, \emph{My numbers, my friends}, Springer-Verlag, New York, 2000,
  Popular lectures on number theory.

\bibitem{Rib04}
P.~Ribenboim, \emph{The little book of bigger primes}, New York-Berlin: Springer-Verlag, Second Edition, (2004).

\bibitem{Slo}
N.~J. A.~Sloane, editor, The On-Line Encyclopedia of Integer Sequences, published electronically at https://oeis.org.

\bibitem{Sol77}
R.~M. Solovay, V.~Strassen, \emph{A fast Monte-Carlo test for primality}, SIAM J. Comp. \textbf{6}, (1977), no.~1, 84--85.

\bibitem{Som09}
L.~Somer, \emph{Lucas pseudoprimes of special types}, Fibonacci Quart. \textbf{47} (2009), no.~3, 198--206.

\bibitem{Suw12}
N.~Suwa, \emph{Some remarks on Lucas pseudoprimes}, Math. J. Okayama Univ. \textbf{54} (2012), 1--32.

\bibitem{Yan04}
S.~Y. Yan, \emph{Primality testing and integer factorization in public-key cryptography}, New York-Berlin: Springer-Verlag, (2004).

\end{thebibliography}
\end{document}